\documentclass[letterpaper, 11pt,  reqno]{amsart}

\usepackage{amsmath,amssymb,amscd,amsthm,amsxtra, esint}
\usepackage{color}
\usepackage{pdfsync}

\usepackage{hyperref}
\usepackage[english, french]{babel}

\allowdisplaybreaks

\newtheorem{thm}{Theorem}[section]
\newtheorem{lem}[thm]{Lemma}

\newtheorem{rem}{Remark}
\newtheorem{cor}[thm]{Corollary}
\newtheorem{defn}{Definition}[section]

\newcommand{\beq}{\begin{equation}}
\newcommand{\eeq}{\end{equation}}
\newcommand{\ben}{\begin{eqnarray}}
\newcommand{\een}{\end{eqnarray}}
\newcommand{\beno}{\begin{eqnarray*}}
\newcommand{\eeno}{\end{eqnarray*}}
\newcommand{\noi}{\noindent}
\newcommand{\les}{\lesssim}

\newcommand{\R}{\mathbb{R}}

\begin{document}

\baselineskip = 15pt

\selectlanguage{english}

\title[2D stationary MHD]
{Liouville-type theorems for the stationary MHD equations in 2D}

\author[W.~Wang]{Wendong Wang}
\address{
Wendong Wang\\
School of Mathematical Sciences\\
Dalian University of Technology\\
Dalian 116024\\
China\\
\&Mathematical Institute\\
University of Oxford\\
 Oxford OX2 6GG, UK}

\email{wendong@dlut.edu.cn}

\author[Y.~Wang]{Yuzhao Wang}
\address{
Yuzhao Wang\\
School of Mathematics\\
University of Birmingham\\
Watson Building\\
Edgbaston\\
Birmingham
B15 2TT}

\email{y.wang.14@bham.ac.uk}

%

\date{\today}
\maketitle


\begin{abstract}
This note is devoted to investigating Liouville type properties of the two dimensional  stationary incompressible
Magnetohydrodynamics equations.
More precisely, under smallness conditions only on the magnetic field, we show that there are no non-trivial solutions
to MHD equations either the Dirichlet integral or some $L^p$ norm of the velocity-magnetic fields
are finite.
In particular, these results generalize the corresponding Liouville type properties for the 2D Navier-Stokes equations,
such as Gilbarg-Weinberger \cite{GW1978} and Koch-Nadirashvili-Seregin-Sverak \cite{KNSS},
to the MHD setting.
\end{abstract}


\setcounter{equation}{0}
\section{Introduction}

In this note, the main concern is the two dimensional (2D) stationary incompressible Magnetohydrodynamics (MHD) equations on the whole plane $\mathbb{R}^2$:
\begin{equation}\label{eq:MHD}
\left\{\begin{array}{llll}
-\mu\Delta u+u\cdot \nabla u+\nabla \pi=b\cdot\nabla b,\\
-\nu\Delta b+u\cdot \nabla b=b\cdot\nabla u,\\
{\rm div }~ u=0,\quad {\rm div }~ b=0,
\end{array}\right.
\end{equation}

\noi
where $u: \R^2 \to \R^2$ and $b: \R^2 \to \R^2$ denote the velocity filed and the magnetic field respectively;  $\mu>0$ is the viscosity coefficient and $\nu>0$ is the resistivity coefficient.
Magnetohydrodynamics is the study of the magnetic properties of electrically conducting fluids, including
 plasmas, liquid metals, etc;  for the physical background
and mathematical theory we refer to Schnack \cite{Sch} and the references therein.
We write the Dirichlet energy as:
\begin{align}
\label{Dirichlet}
D(u,b)=\int_{\R^2}|\nabla u|^2+|\nabla b|^2 \,dx,
\end{align}

\noi
which plays an important role in the Liouville theory concerning MHD equations \eqref{eq:MHD}.

When $b=0$ and $\mu =1$, the MHD equation \eqref{eq:MHD} reduces to the standard 2D Navier-Stokes (NS) equations,
\begin{equation}\label{eq:NS}
\left\{\begin{array}{llll}
-\Delta u+u\cdot \nabla u+\nabla \pi=0,\\
{\rm div }~ u=0,
\end{array}\right.
\end{equation}

\noi
for which Liouville properties are well understood.
For instance,
Gilbarg-Weinberger \cite{GW1978} proved that there are only constant solutions to \eqref{eq:NS}
provided the Dirichlet energy is finite, that is
\[
\int_{\R^2} |\nabla u|^2  \,dx < \infty.
\]
Their proof relies on the fact that the vorticity of the 2D NS equations \eqref{eq:NS} satisfies a nice elliptic equation,
to which the maximum principle applies.
To be more precise, for a solution $u$ to \eqref{eq:NS}, define $w = \partial_2 u_1 - \partial_1 u_2$ to be its vorticity.
Then $w$ solves the following elliptic equation
\[
\Delta w - u \cdot \nabla w = 0,
\]
which satisfies the maximal principle.
The assumption on boundedness of the Dirichlet energy can be relaxed to $\nabla u\in L^p(\mathbb{R}^2)$ with some $p\in (\frac65,2]$, see \cite{BFZ2013}.
As a different type of Liouville property for the 2D NS,
Koch-Nadirashvili-Seregin-Sverak \cite{KNSS} showed
that any bounded solution to \eqref{eq:NS} is trivial solution, say $u \equiv C$,
as a byproduct of their results on the non-stationary case.
In \cite{KNSS} they exploited the maximum principle of a parabolic type, see also a note of Koch \cite{Koch}.
Recently, it was extended to the
case of generalized Newtonian fluids,
where the viscosity is a function depending on the shear
rate in \cite{Fu2012Liou,Fuguo2012}. See also \cite{ZG2015} for a similar result for $u\in L^p(\mathbb{R}^2)$ with $p>1$ on the generalized Newtonian fluid.
Other types of Liouville properties for the stationary Navier-Stokes equation on the plane were also extensively studied,
such as under the growth condition $\limsup |x|^{-\alpha}|u(x)|<\infty$ as
$|x|\rightarrow\infty$ for some $\alpha >0$, see \cite{FZ2011,BFZ2013};
existence and asymptotic behavior of solutions in an exterior domain, see \cite{GNP1997,Ru2009,Ru2010,GG2011,PR2012,KPR2014,DI2017}.
For more references on Liouville theorems of (\ref{eq:NS}), we refer to \cite{Galdi,Fu2012exist,ZG2013,JK2014,Fu2018} and the references therein.


\vspace{2mm}

Before proceed with our man result, we define the weak solution to the MHD system (\ref{eq:MHD}).

\begin{defn}
We say that $(u,b)$ is a weak solution to the 2D MHD equations (\ref{eq:MHD})
in a domain $\Omega\subset \mathbb{R}^2$ provided that:

\vspace{2mm}

\begin{itemize}

\item[(i).] $u,b\in L_{loc}^s(\Omega)$ for some $s\geq 2$;\\
\item[(ii).] $\textup{div}\, u=0$ and $\textup{div}\, b=0$, in the weak sense;\\
\item[(iii).] $(u,b)$ satisfies the following system

\beno
\mu\int_{\Omega}u\cdot \triangle\phi \,dx+\int_{\Omega}(u\cdot \nabla\phi)\cdot u \,dx=\int_{\Omega}(b\cdot \nabla\phi)\cdot b \,dx
\eeno
and
\beno
\nu\int_{\Omega}b\cdot \triangle\phi \,dx+\int_{\Omega}(u\cdot \nabla\phi)\cdot b \,dx=\int_{\Omega}(b\cdot \nabla\phi)\cdot u \,dx
\eeno
for all $\phi\in C_0^{\infty}(\Omega)$ with $\phi=(\phi_1,\phi_2)$ and $\textup{div}\,\phi=0.$

\end{itemize}
In what follows, we shall take $\Omega = \R^2$ unless otherwise specified.

\end{defn}

A natural question is whether the above mentioned Liouville properties
hold for the 2D stationary MHD equations \eqref{eq:MHD}.
One may try to modify the arguments in \cite{GW1978}
or \cite{KNSS} for Navier Stokes equations \eqref{eq:NS} to that of
MHD \eqref{eq:MHD}.
However, this is not the case.
For instance, due to the presence of the magnetic fields,
the maximum principle doesn't hold for the vorticity of the MHD equations.
Therefore,
 Gilbarg-Weinberger's argument fails to apply to the 2D MHD equation.
Nevertheless, we step forward in this direction and provide
 positive answers to this question by assuming smallness of the magnetic fields.

Our first main result is as follows,

\begin{thm}\label{thm1}
Let $(u,b)$ be a weak solution of the 2D MHD equations (\ref{eq:MHD}) defined over the entire plane.
Assume that $D(u,b)\leq D_0<\infty$ and
\beno
\|b\|_{L^1(\mathbb{R}^2)}D_0^{\frac12}\leq C_*\min\{\mu\nu, \mu^{\frac12}\nu^{\frac32}\},
\eeno
where $C_*$ is an absolute constant.
Then $u$ and $b$ are constants.
\end{thm}

\begin{rem}\rm
\label{rmk-smooth}
Similar analysis as Galdi in \cite{Galdi},
for any  weak solutions $(u,b)$ to the stationary MHD equations \eqref{eq:MHD},
if $u,b\in L^s_{loc}(\Omega; \R^2)$ with some $s>2$,
then $u,b\in W^{1,2}_{loc}(\Omega; \R^2)$ and $u,b$ are smooth as a consequence of the regularity property of Stokes equations.
For more details, we refer readers to { \cite[Chapter IX]{Galdi}.}
Therefore, the weak solutions to \eqref{eq:MHD} are indeed smooth under the conditions of Theorem \ref{thm1}.
\end{rem}

\begin{rem}\rm
We stress that smallness conditions only apply to the magnetic field $b$.  Note that if $(u,b)$ be a solution of (\ref{eq:MHD}), then
\beno
u^\lambda(x)\doteq\lambda u(\lambda x),\quad b^\lambda(x)\doteq\lambda b(\lambda x)
\eeno
is also a solution of (\ref{eq:MHD}). The quantities $\|b\|_{L^1(\mathbb{R}^2)}\|\nabla u\|_{L^2(\mathbb{R}^2)}$ or 
$\|b\|_{L^1(\mathbb{R}^2)}\|\nabla b\|_{L^2(\mathbb{R}^2)}$ are invariant under the natural scaling.
By the way, our proof doesn't appeal to the special structure of the vorticity equation
of the 2D NS equations as \cite{GW1978} did,
and so it is more robust  in extending to more general settings.
\end{rem}

Motivated by  \cite{KNSS} and \cite{ZG2015}, our second result is concerned with the Liouville property for $L^p$ solutions,

\begin{thm}\label{thm2}
Let $(u,b)$ be a weak solution of the 2D MHD equations (\ref{eq:MHD}) defined over the entire plane. Then $u$, $b$ are constants if one of the following conditions hold:
\begin{enumerate}
\item $u,b\in L^p(\mathbb{R}^2,\mathbb{R}^2)$ for some $p\in(2, 6]$;
\item $\|u\|_{L^p(\mathbb{R}^2)}+\|b\|_{L^p(\mathbb{R}^2)}\leq L<\infty$ for some $p\in(6, \infty]$, and
there exists an absolute constant $C_*$ such that $
\|b\|_{L^1(\R^2)}L^{\frac{p}{p-2}}\leq C_*\min\{\mu\nu, \mu^{\frac12}\nu^{\frac32}\}$.
\end{enumerate}
\end{thm}

\begin{rem}\rm
The condition $p>2$ is to ensure the regularity of weak solution of (\ref{eq:MHD}). When $p\in (2,6]$, no smallness conditions are needed, that is to say there are no non-trivial $L^p$ solutions to \eqref{eq:MHD}. However, it is different when $p>6$. The main difference comes from a simple fact: the estimate of the nonlinear term $u\cdot\nabla u$ or $b\cdot\nabla b$, and $R^{-1}\int_{B_R\setminus B_{R/2}}|b|^3dx=o(R)$ as $R\rightarrow\infty$ if $b\in L^p(\R^2)$ satisfying $p\leq 6$(see Section 4.1,4.2).  When $p\in (6,\infty]$, we need to assume that the scaling invariant norms  $\|b\|_{L^1(\R^2)}\|u\|_{L^p(\R^2)}^{\frac{p}{p-2}}$ or $\|b\|_{L^1(\R^2)}\|b\|_{L^p(\R^2)}^{\frac{p}{p-2}}$ are sufficiently small. Moreover,
the above result generalizes the corresponding theorems for the Navier Stokes equation (\ref{eq:NS}) in \cite{KNSS} or \cite{ZG2015} to the setting of MHD equations.
\end{rem}

%

\section{Preliminaries}

In this section, we prepare some preliminary lemmas that we shall rely on.
Throughout this article, $C(a_1,\cdots,a_k)$ denotes a constant depending on $a_1,\cdots, a_k$, which may be different from line to line.
We denote the ball with centre $x_0$ of radius $R$ by $B_R(x_0)$.
If $x_0=0$, we simply write $B_R = B_R(0)$.
Let a radially decaying smooth $\eta(x)$ be a test function such that
\begin{align*} \eta(x)=\left\{
\begin{aligned}
&1,\quad x\in B_1,\\
&0, \quad x\in B_2^c.
\end{aligned}
\right. \end{align*}
and
let
\begin{align}
\label{etaR}
\eta_R (x) = \eta \left(\frac{x}{R} \right)
\end{align}
for $R>0$.
One notices that $|\nabla^k \eta_R| \le \frac{C}{R^k}$.

%
%
%
%
%
%

Let us recall a result of Gilbarg-Weinberger in \cite{GW1978} about the decay of functions with finite Dirichlet integrals.
\begin{lem}[Lemma 2.1, 2.2, \cite{GW1978}]
\label{lem:GW}
Let a $C^1$ vector-valued function $f(x)=(f_1,f_2)(x)=f(r,\theta)$ with $r=|x|$ and $x_1=r\cos\theta$.  There holds finite Dirichlet integral in the range $r>r_0$, that is
\beno
\int_{r>r_0}|\nabla f|^2\,dxdy<\infty.
\eeno
Then, we have
\beno
\lim_{r\rightarrow\infty} \frac{1}{\ln r}\int_0^{2\pi}|f(r,\theta)|^2d\theta=0.
\eeno
and furthermore,
there is an increasing sequence $\{r_n\}$ with $r_n\in (2^n, 2^{n+1})$, such that
\beno
\lim_{n\rightarrow \infty}\frac{|f(r_n,\theta)|^2}{\ln r_n}=0,
\eeno
uniformly in $\theta.$
\end{lem}

\noi
If, furthermore, we assume $\nabla f \in L^p(\R^2)$ for some $2< p < \infty$,
then the above decay property can be improved to be point-wise uniformly.
More precisely, we have

\begin{lem}[Theorem II.9.1 \cite{Galdi}]
\label{unidec}
Let $\Omega \subset \R^2$ be an exterior domain and let
\[
\nabla f \in L^2 \cap L^p (\Omega),
\]
for some $2< p < \infty$.
Then
\[
\lim_{|x| \to \infty} \frac{|f(x)|}{\sqrt{ \ln (|x|)}} = 0,
\]
uniformly.
\end{lem}

We also need a Giaquinta's iteration lemma \cite[Lemma 3.1]{G83}, also see a proof in \cite[Lemma 8]{ChYa}.

\begin{lem}[Lemma 3.1 \cite{G83}]
\label{iter}
Let $f(r)$ be a non-negative bounded function on $[R_0,R_1] \subset \R_+$. If there are negative constants $A,B,D$ and positive exponents $b<a$ and a parameter $\theta \in (0,1)$ such that for all $R_0 \le \rho < \tau \le R_1$
\[
f(\rho) \le \theta f(\tau) + \frac{A}{(\tau-\rho)^{a}} + \frac{B}{(\tau-\rho)^{b}} + D,
\]
then for all $R_0 \le \rho < \tau \le R_1$
\[
f(\rho) \leq C(a,\theta) \left[\frac{A}{(\tau-\rho)^{a}} + \frac{B}{(\tau-\rho)^{b}} + D\right].
\]
\end{lem}

\section{Proof of Theorem \ref{thm1}}


For a solution of $(u,b)$ to \eqref{eq:MHD}, consider the vorticity $w=\partial_2u_1-\partial_1 u_2$ and the current density  $h=\partial_2b_1-\partial_1 b_2$.
It is easy to check that $(w,h)$ satisfies
\begin{equation}\label{eq:MHDv-2D}
\left\{\begin{array}{llll}
-\mu\Delta w+u\cdot \nabla w=b\cdot\nabla h,\\
-\nu\Delta h+u\cdot \nabla h=b\cdot\nabla w+ H
\end{array}\right.
\end{equation}
where
\begin{align}
\label{H}
H=2\partial_2b_2(\partial_2u_1+\partial_1 u_2)+2\partial_1u_1(\partial_2b_1+\partial_1 b_2)
\end{align}

\noi
One crucial step is to get the higher regularity estimates of the solutions of (\ref{eq:MHD}).
Different from the argument in \cite{GW1978},
we have to exploit something new to overcome the obstacle due to the lack of maximum principle for the 2D MHD equations.
Before proceeding with the proof of Theorem \ref{thm1}, we prove the following smoothing property for the solution of \eqref{eq:MHDv-2D}.

\begin{lem}\label{lem:D2u}	
Let the vorticity $w$ and the current $h$ as in the MHD equations (\ref{eq:MHDv-2D}) with finite Dirichlet integral, i.e.
$D(u,b)<\infty$.
Then, we have
\ben\label{D2u}
\int_{\mathbb{R}^2}|\nabla w|^2+|\nabla h|^2\,dx<\infty;
\een
and furthermore, under the polar coordinate $x = r \cos \theta$ and $ y = r\sin \theta$, we have
\ben\label{decayu}
\lim_{r\rightarrow \infty}\frac{|u(r,\theta)|^2}{\ln r}+\frac{|b(r,\theta)|^2}{\ln r}=0
\een
uniformly in $\theta.$
\end{lem}

\begin{proof}
We assume $\mu=\nu=1$ without loss of generality. Choose a cut-off function $\phi(x)\in C_0^\infty(B_R)$ with $0\leq \phi\leq 1$ satisfying the following two properties:

\begin{enumerate}

\item[i).]
$\phi$ is radially decreasing and satisfies
\begin{align*} \phi(x)=\phi(|x|)=\left\{
\begin{aligned}
&1,\quad |x|\leq \rho,\\
&0, \quad |x|\geq\tau,
\end{aligned}
\right. \end{align*}
where $0<\frac{R}{2}\leq \rho<\tau\leq R$;

\item[ii).]
 $|\nabla\phi|(x)\leq \frac{C}{\tau-\rho}$ for all $x\in \mathbb{R}^2$.

\end{enumerate}



\noi
 Multiplying both sides of
\eqref{eq:MHDv-2D} by $\phi^2 w$ and $\phi^2 h$ respectively and then
integrating over $\R^2$ to get
\[
\int_{\mathbb{R}^2} \phi^2 |\nabla w |^2 \,dx = - \int_{\mathbb{R}^2} \nabla w \cdot \nabla (\phi^2) w \,dx - \int_{\mathbb{R}^2} u \cdot \nabla w \phi^2 w \,dx + \int_{\mathbb{R}^2} b \cdot \nabla h \phi^2 w \,dx
\]
and
\[
\int_{\mathbb{R}^2} \phi^2 |\nabla h|^2 \,dx = - \int_{\mathbb{R}^2} \nabla h \cdot\nabla (\phi^2)  h \,dx - \int_{\mathbb{R}^2}  u \cdot \nabla h \phi^2 h \,dx + \int_{\mathbb{R}^2} b \cdot \nabla w \phi^2 h \,dx + \int_{\mathbb{R}^2} H \phi^2 h \,dx.
\]
By noticing the cancelation
\beno
\displaystyle \int_{\mathbb{R}^2} b \cdot \nabla h \phi^2 w \,dx + \int_{\mathbb{R}^2} b \cdot \nabla w \phi^2 h \,dx = - \int_{\mathbb{R}^2} b\cdot \nabla (\phi^2) hw \,dx,
\eeno
and then applying integration by parts, we arrive at
\begin{align}
\label{eq:vorticity equ}
& \hspace{-6mm} \int_{\mathbb{R}^2} \phi^2 |\nabla w |^2 \,dx + \int_{\mathbb{R}^2} \phi^2 |\nabla h|^2 \,dx \notag\\
& = -  \int_{\mathbb{R}^2} \nabla w  \cdot \nabla (\phi^2) w  \,dx
-  \int_{\mathbb{R}^2}  \nabla h \cdot \nabla (\phi^2)  h \,dx+ \frac12 \int_{\mathbb{R}^2} u \cdot \nabla  (\phi^2) w^2 \,dx  \notag\\
& \hspace{4mm}  + \frac12 \int_{\mathbb{R}^2} u \cdot \nabla  (\phi^2)  h^2 \,dx - \int_{\mathbb{R}^2} b\cdot \nabla (\phi^2) hw \,dx + \int_{\mathbb{R}^2} H h\phi^2 \,dx \notag\\
&\doteq I_1 + \cdots+ I_6.
\end{align}

\noi
In what follows we shall estimate $I_j$ for $j = 1,2,\cdots, 6$ one by one.

For the term $I_1$, by H\"older's inequality and \eqref{Dirichlet} we have
\begin{align*}
I_1 & \le \frac{C}{\tau-\rho} \| \nabla w \|_{L^2 (B_\tau)} \| w \|_{L^2 (B_\tau )} \\
& \leq \frac18 \int_{B_{\tau}} |\nabla w|^2 \,dx + \frac {C}{(\tau-\rho)^2},
\end{align*}
and similarly
\beno
I_2 \le \frac18 \int_{B_{\tau}} |\nabla h|^2 \,dx + \frac {C}{(\tau-\rho)^2} .
\eeno

\noi

For the terms $I_3,\cdots,I_5$, it only needs to consider $I_5$ since other terms can be treated similarly.
Let
\beno
\bar{f}(r)=\frac{1}{2\pi}\int_0^{2\pi}f(r,\theta)d\theta,
\eeno
then by Wirtinger's inequality (for example, see Ch II.5 \cite{Galdi}) we have
\ben\label{eq:Wirtinger}
\int_0^{2\pi}|f-\bar{f}|^2 \, d\theta\leq \int_0^{2\pi}|\partial_\theta f|^2d\theta.
\een
By H\"older inequality, (\ref{eq:Wirtinger}) and Lemma \ref{lem:GW} we have
\begin{align*}
I_5&\leq \left|\int_{\mathbb{R}^2}
 h\,w \,b\cdot\nabla\phi \,dx\right|\\
&\leq \left|\int_{\mathbb{R}^2}h\,w\,(b-\bar{b})\cdot\nabla\phi \,dx\right|+\left|\int_{\mathbb{R}^2}h\,w\,\bar{b}\cdot\nabla\phi \,dx\right|\\
&   \leq \frac{ C}{\tau-\rho}\left( \int_{B_\tau}w^4\right)^{\frac14}
\left( \int_{B_\tau}h^4\right)^{\frac14}\left( \int_{\frac{R}2<|r|<R} \int_0^{2\pi}|b(r,\theta)- \bar{b} |^2 \, d\theta \,r dr \right)^{\frac12}\\
&\hspace{8mm} + \frac{C}{\tau-\rho}  \int_{B_\tau\setminus B_{\frac{R}2}} |wh| \left( \int_0^{2\pi} |b (r,\theta)|^2 \, d\theta \right)^{\frac12} \,dx\\
&   \leq \frac{ CR}{\tau-\rho}\left( \int_{B_\tau}w^4\right)^{\frac14}\left( \int_{B_\tau}h^4\right)^{\frac14}\left( \int_{\frac{R}2<|r|<R}\frac{1}{r^2}  \int_0^{2\pi}|\partial_\theta b|^2d\theta \,rdr\right)^{\frac12}\\
&\hspace{8mm} +C \frac{(\ln R)^{\frac12}}{\tau-\rho} \int_{B_\tau}(w^2+h^2)\,dx.
\end{align*}
Using the following Poincar\'{e}-Sobolev inequality(see, for example, Theorem 8.11 and 8.12 \cite{LL})
\ben
\label{eq:poincare-sobolev}
\|w\|_{L^4(B_\tau)}\leq C \|\nabla w\|_{L^2(B_\tau)}^{\frac12}\|w\|_{L^2(B_\tau)}^{\frac12}+C\tau^{-1}\|w\|_{L^2(B_\tau)},
\een
we obtain
\beno
I_5& \leq &\frac{ CR}{\tau-\rho}\left( \int_{B_\tau}|\nabla w|^2+|\nabla h|^2\right)^{\frac12}+\frac{ CR\tau^{-2}}{\tau-\rho}+ \frac{C\sqrt{\ln R}}{\tau-\rho},
\eeno
where we used the boundedness of Dirichlet integral.
Thus
\begin{align*}
I_3+I_4+I_5 \leq   \frac18 \int_{B_{\tau}} |\nabla h|^2+ |\nabla w|^2\,dx+\frac{C R\tau^{-2}}{\tau-\rho}+ \frac{C\sqrt{\ln R}}{\tau-\rho}+\frac{C R^2}{(\tau-\rho)^2}.
\end{align*}

For the term $I_6$, using (\ref{eq:poincare-sobolev}) again we get
\begin{align*}
I_6 & = \int_{\mathbb{R}^2}\phi h H \,dx \leq C \|\nabla b\|_{L^4(B_\tau)}^2\|w\|_{L^2(B_\tau)}\\
&\leq   \frac18 \int_{B_{\tau}} |\nabla^2 b|^2\,dx+C(1+\tau^{-2}).
\end{align*}
Moreover, due to $\nabla^\bot=(\partial_2,-\partial_1)^\top$ and div$\,u=0$, there holds
\beno
\triangle u=\nabla^\bot (\partial_2u_1-\partial_1u_2)=\nabla^\bot  w.
\eeno
Thus by integration by parts we have
\begin{align}
\label{eq:embeding inequ}
& \hspace{-4mm} \int_{\mathbb{R}^2}\phi^2|\nabla^2 u|^2\,dx \notag \\
& \leq \frac43 \int_{\mathbb{R}^2}\phi^2|\triangle u|^2\,dx+\frac{C}{(\tau-\rho)^2}\int_{B_{\tau}}|\nabla u|^2\,dx\nonumber\\
& \leq  \frac43 \int_{B_{\tau}}\phi^2|\nabla w|^2\,dx+ \frac{C}{(\tau-\rho)^2}.
\end{align}

\noi
Collecting the estimates $I_1,\cdots,I_6$, by (\ref{eq:embeding inequ}) we get
\begin{align*}
& \hspace{-4mm}\frac34 \int_{B_\rho}|\nabla^2 u|^2+|\nabla^2 b|^2\,dx \\
&\leq \frac12 \int_{B_{\tau}} |\nabla^2 b|^2+ |\nabla^2 u|^2\,dx\\
& \hspace{6mm}+C(1+\tau^{-2})+\frac{ CR\tau^{-2}}{\tau-\rho}+ \frac{C\sqrt{\ln R}}{\tau-\rho}+\frac{ CR^2}{(\tau-\rho)^2}+\frac {C}{(\tau-\rho)^2}.
\end{align*}
Then by applying Lemma \ref{iter}, we obtain
\beno
\int_{B_{R/2}}|\nabla^2 u|^2+|\nabla^2 b|^2\,dx\leq C R^{-2}+C \frac{\sqrt{\ln R}}{R}+C.
\eeno
Finally, by taking $R\rightarrow\infty$, we arrive at \eqref{D2u}.

Now we turn to the proof of \eqref{decayu}.
By Gagliardo-Nirenberg inequality, one notices that
\[
\|\nabla u \|_{L^4(\R^2)} \leq C \| \nabla u \|_{L^2(\R^2)}^{\frac12} \| \nabla^2 u \|_{L^2(\R^2)}^{\frac12} .
\]
Then \eqref{decayu} follows from \eqref{D2u} and Lemma \ref{unidec}.
Therefore, the proof is complete.
\end{proof}

\begin{rem}\rm
Lemma \ref{lem:D2u} roughly says that by assuming the boundedness of the Dirichlet integral \eqref{Dirichlet}, $L^2$ norm of the gradient, one can bound the second order derivatives, \eqref{D2u}.
This is a manifestation of the smoothing effect,
which will be used as a substitution of the maximal principle in \cite{GW1978}.
Please also note that the assumptions on the magnetic field $b$ in Lemma \ref{lem:D2u} holds automatically for the Navier-Stokes equation since then $b=0$.
Therefore, in this perspective, the smoothing effect exploited by Lemma \ref{lem:D2u} is more robust than the maximal principle used in
\cite{GW1978}.
\end{rem}

Now we are ready to demonstrate the proof of Theorem \ref{thm1}.

\begin{proof}[Proof of Theorem \ref{thm1}.]
Making the inner product with $\eta_R^2{w}$ on both sides of the equation $(\ref{eq:MHDv-2D})_1$, and $\eta_R^2{h}$ on both sides of the equation $(\ref{eq:MHDv-2D})_2$,  we have
\begin{align}
\label{D1u}
\mu\int_{B_{R}} & |\nabla w|^2+\nu\int_{B_{R}}|\nabla h|^2\,dx\notag\\
&\leq  \frac{C}{R^2}\left(\int_{B_{2R}\setminus B_{R}}| {w}|^2+| {h}|^2\,dx\right)+ \frac{C}{R}\left(\int_{B_{2R}\setminus B_{R}}|u|| {w}|^2+|h|^2|u|+|b||h||w|\,dx\right)\notag\\
& \hspace{4mm} +\left|\int_{B_{2R}}H h\eta_R^2\,dx\right| \notag\\
& \doteq I_1+I_2+I_3,
\end{align}
where $\eta_R$ is as in \eqref{etaR} and
$H$ is as in \eqref{H}.

Terms $I_1$ and $I_2$ are easy to estimate.
By $D(u,b)\leq D_0$ and (\ref{decayu}) we have
\begin{align}
\label{I-1I-2}
|I_1|+ |I_2|\leq CR^{-2} +\frac{C}{R}\sqrt{\ln R}.
\end{align}

It remains to bound $I_3$. In what follows, we may assume that $I_3 = 2\partial_2 b_2 \partial_2 u_1$ since the treatments  for other terms are similar.
\begin{align*}
|I_{3}| &= \left| \int_{B_{2R}}\eta_R^2h\partial_2b_2\partial_2u_1\,dx \right| \\
&\leq \left|\int_{B_{2R}}b_2\partial_2\partial_2u_1h\eta_R^2\,dx\right| +\left|\int_{B_{2R}}b_2\partial_2u_1\partial_2h\eta_R^2\,dx\right| +\left|\int_{B_{2R}}b_2\partial_2u_1h\partial_2\eta_R^2\,dx\right|\\
&\leq \left(\int_{B_{2R}}|b_2h\eta_R|^2\,dx\right)^{1/2}\left(\int_{B_{2R}}|\partial_2\partial_2u_1\eta_R|^2\,dx\right)^{1/2}\\
& \hspace{4mm}
+\left(\int_{B_{2R}}|\partial_2h\eta_R|^2\,dx\right)^{1/2}\left(\int_{B_{2R}}|b_2\partial_2u_1\eta_R|^2\,dx\right)^{1/2}+\frac{C}{R}\sqrt{\ln R}.
\end{align*}
where we used $D(u,b)\leq D_0$ and (\ref{decayu}).
For the first factor of the first term,
due to the Gagliardo-Nirenberg inequality we have
\begin{align*}
& \hspace{-5mm}  \left(\int_{B_{2R}}|b_2h\eta_R|^2\,dx\right)^{1/2}\\
& \leq \|b\|_{L^4(\mathbb{R}^2)}\|h\|_{L^4(\mathbb{R}^2)}\\
& \leq C\|b\|_{L^1(\mathbb{R}^2)}^{\frac12}\|\nabla^2b\|_{L^2(\mathbb{R}^2)}^{\frac12}\|h\|_{L^2(\mathbb{R}^2)}^{\frac12}\|\nabla h\|_{L^2(\mathbb{R}^2)}^{\frac12}\\
&\leq C \left( \|b\|_{L^1(\mathbb{R}^2)} D_0^{\frac12} \right)^{\frac12}\|\nabla h\|_{L^2(\mathbb{R}^2)},
\end{align*}
where we used (\ref{D2u}), $D(u,b)\leq D_0$, and (\ref{eq:embeding inequ}).
Similarly, we have
\beno
\left(\int_{B_{2R}}|b_2\partial_2u_1|^2\,dx\right)^{1/2} \le  C \left( \|b\|_{L^1(\mathbb{R}^2)} D_0^{\frac12} \right)^{\frac12} \|\nabla h\|_2^{\frac12}\|\nabla w\|_2^{\frac12}.
\eeno
Hence, by letting $R\rightarrow\infty$, we conclude
\beno
I_{3}\leq C \left( \|b\|_{L^1(\mathbb{R}^2)} D_0^{\frac12} \right)^{\frac12} \Big( \|\nabla h\|_2\|\nabla w\|_2+  \|\nabla h\|_2^{\frac32}\|\nabla w\|_2^{\frac12} \Big).
\eeno
By choosing $\|b\|_{L^1(\mathbb{R}^2)} D_0^{\frac12} $ small enough, we arrive at
\begin{align}
\label{I-3}
I_{3}\leq \frac{\mu}{16}\|\nabla h\|_2^2+\frac{\nu}{16}\|\nabla w\|_2^2.
\end{align}
For instance, one may choose $
  \|b\|_{L^1(\mathbb{R}^2)} D_0^{\frac12} \leq C_*\min\{\mu\nu, \mu^{\frac12}\nu^{\frac32}\},
$
where $C_*$ is an absolute constant.

By collecting (\ref{D1u}), \eqref{I-1I-2}, and \eqref{I-3},  we finally get
\beno
\mu\int_{B_R}|\nabla w|^2+\nu\int_{B_R}|\nabla h|^2\,dx
\leq CR^{-2} +\frac{C}{R}\ln R + \frac{\mu}{16}\|\nabla w\|_2^2+\frac{\nu}{16}\|\nabla h\|_2^2.
\eeno
Consequently, letting $R\rightarrow\infty$, we conclude that
\beno
\nabla w= \nabla h=0.
\eeno
It follows that both $w$ and $h$ are constants.
Due to $D(u,b)\leq D_0$, we conclude that  $w=0$ and $h=0$.
Finally, since div$\,u=0$ and div$\,b=0$, it follows that
 $u$ and $b$ are constants.
Furthermore, one notices that $b=0$ since $b\in L^1$.
Thus the proof is finished.
\end{proof}

\section{Proof of Theorem \ref{thm2}}

In this section, the proof
relies on a Giaquinta's iteration lemma \cite[Lemma3.1]{G83}.
We assume that $\mu=\nu=1$ for simplicity.
The proof is split into four cases: $3\le p \le 6$, $2 \le p < 4$, $6< p < \infty $, and $p = \infty$.
The arguments for the former two cases are similar, the main point of which is to establish a gradient estimate;
while the later two cases appeal to estimates involving second order derivatives.
We shall give full detailed proofs for the first and third cases,
and indicate where modification is needed to treat the second and fourth cases.

Let us start with the first case.

\subsection{Case $3\le p \le 6$}
\label{3to6}
At first, we fix a $R\in \R_+$ and the cut-off function $\phi(x)\in C_0^\infty(B_R)$ as in the previous section.
\noi
By the choice of the parameters there holds
\beno
0<\frac{R}{2}<\frac23\tau<\frac{3}{4}R\leq\rho<\tau\leq R.
\eeno

Due to Theorem III 3.1 in \cite{Galdi}, there exists a constant $C(s)$ and a vector-valued function $\bar{w}: B_\tau\setminus B_{\frac23\tau}\rightarrow \mathbb{R}^2$ such that
$\bar{w}\in W^{1,s}_0(B_\tau\setminus B_{\frac23\tau})$, and
$\nabla\cdot \bar{w}(x)=\nabla_x\cdot[\phi(x) {u}(x)]$. Moreover,  we get
\begin{align}
\label{esti-ws}
\int_{B_\tau\setminus B_{\frac23\tau}}|\nabla \bar{w}(x)|^s\,dx
\leq C(s)\int_{B_\tau\setminus B_{\frac23\tau}}|\nabla\phi\cdot { u}|^s\,dx.
\end{align}
We thus can extend $\bar{w}$ to the whole space $\mathbb{R}^2$, which vanishes outside of the domain $B_\tau.$

\begin{proof}[Proof of Theorem \ref{thm2}: case $3\le p\le 6$]

Making the inner products $(\phi {u}-\bar{w})$ and $\phi b$ on both sides of the equation (\ref{eq:MHD}), by $\nabla\cdot \bar{w}=\nabla\cdot[\phi{u}]$ we have
\begin{align*}
&\hspace{-8mm}\int_{B_\tau}\phi|\nabla u|^2\,dx
\\&= -\int_{B_\tau}\nabla\phi\cdot\nabla u \cdot {u} \,dx+\int_{B_\tau\setminus B_{\frac23\tau}}\nabla \bar{w}:\nabla u  \,dx
-\int_{B_\tau}u\cdot\nabla u \cdot \phi {u} \,dx\\
& \hspace{5mm}+\int_{B_\tau\setminus B_{\frac23\tau}}u\cdot\nabla u \cdot \bar{w} \,dx +\int_{B_\tau}b\cdot\nabla b \cdot \phi {u} \,dx-\int_{B_\tau\setminus B_{\frac23\tau}}b\cdot\nabla b \cdot \bar{w} \,dx \\
&=  I_1+\cdots+I_6,
\end{align*}
and
\begin{align*}
&\hspace{-8mm}\int_{B_\tau}\phi|\nabla b|^2\,dx
\\&= -\int_{B_\tau}\nabla\phi\cdot\nabla b\cdot {b} \,dx
-\int_{B_\tau}u\cdot\nabla b \cdot \phi {b} \,dx
+\int_{B_\tau}b\cdot\nabla u \cdot \phi {b} \,dx \\
&= I_1'+I_2'+I_3'.
\end{align*}

For the term $I_1$, by H\"{o}lder inequality  we have
\beno
|I_1|\leq \frac{C}{\tau-\rho}\left(\int_{B_\tau}|\nabla u|^2\,dx\right)^{\frac12}\left(\int_{B_R\setminus B_{R/2}}| {u}|^2\,dx\right)^{\frac12}.
\eeno

\noindent
For the term $I_2$, H\"{o}lder inequality and (\ref{esti-ws}) imply that
\begin{align*}
|I_2|&\leq C\left(\int_{B_\tau}|\nabla u|^2\,dx\right)^{\frac12}\|\nabla \bar{w}\|_{L^{2}(B_\tau\setminus B_{\frac23\tau})}\\
&\leq  \frac{C}{\tau-\rho}\|\nabla u\|_{L^2(B_\tau)} \|{u}\|_{L^2(B_R\setminus B_{R/2})}.
\end{align*}

\noi
By integration by parts, it is easy to find that
%
\beno
|I_3| \leq \frac{C}{\tau-\rho}\|{u}\|_{L^3(B_R\setminus B_{R/2})}^3.
\eeno

\noi
For the term $I_4$, integration by parts leads to
\[
I_4 = -\int_{B_\tau\setminus B_{\frac23\tau}} u \cdot \nabla \bar{w} \cdot u \, dx.
\]

\noi
Then in view of \eqref{esti-ws} we find
\begin{align*}
|I_4| & \le  \|{u}\|_{L^3(B_R\setminus B_{R/2})}^2 \| \nabla \bar{w}\|_{L^3}\\
& \le \frac{C}{\tau-\rho} \|{u}\|_{L^3(B_R\setminus B_{R/2})}^3.
\end{align*}

\noi
For the term $I_5$, we need a cancellation with $I'_3$. More precisely,
\[
I_5 + I_3' = - \int_{B_\tau} \big( b \otimes b \big) : \big( \nabla \phi \otimes u \big),
\]
and 
it follows that

\beno
|I_5 + I'_3| \leq \frac{C}{\tau-\rho}
\Big( \|{u}\|_{L^3(B_R\setminus B_{R/2})}^3+  \|{b}\|_{L^3(B_R\setminus B_{R/2})}^3 \Big).
\eeno

\noi
The treatment for $I_6$ is similar to $I_4$ and

\[
|I_6| \le  \frac{C}{\tau-\rho} \|{b}\|_{L^3(B_R\setminus B_{R/2})}^2 \|{u}\|_{L^3(B_R\setminus B_{R/2})}.
\]

\noi
For the terms $I'_1$ and $I_2'$, similar as $I_1$ and $I_3$ respectively, we find

\beno
|I_1'|+|I_2'|\leq \frac{C}{\tau-\rho}\|\nabla b\|_{L^2(B_\tau)} \|{b}\|_{L^2(B_R\setminus B_{R/2})}+  \frac{C}{\tau-\rho}
\Big(\|{u}\|_{L^3(B_R\setminus B_{R/2})}^3+  \|{b}\|_{L^3(B_R\setminus B_{R/2})}^3 \Big).
\eeno

By setting
\begin{align}
\label{f}
f(r) = \int_{B_r} |\nabla u|^2 + |\nabla b|^2 \,dx,
\end{align}
collecting the above estimates we have
\begin{align*}
f(\rho) & \leq \frac12 f(\tau) + \frac{C}{\tau-\rho}
\Big(\|{u}\|_{L^3(B_R\setminus B_{R/2})}^3+  \|{b}\|_{L^3(B_R\setminus B_{R/2})}^3 \Big)\\
& \hspace{18mm}+ \frac{C}{(\tau-\rho)^2}
\Big(\|{u}\|_{L^2(B_R\setminus B_{R/2})}^2+  \|{b}\|_{L^2(B_R\setminus B_{R/2})}^2 \Big).
\end{align*}

\noi
Now we apply Lemma \ref{iter} with $R_0 = \frac{3R}4$ and $R_1 = R$ to obtain

\begin{align}
\label{uniloc}
& \hspace{-6mm}\int_{B_{R/2}}|\nabla u|^2+|\nabla b|^2\,dx \notag \\
&\leq  \frac{C}{R^2} \Big(\|{u}\|_{L^2(B_R\setminus B_{R/2})}^2 +  \|{b}\|_{L^2(B_R\setminus B_{R/2})}^2 \Big) \notag \\
 & \hspace{16mm} + \frac{C}{R} \Big(\|{u}\|_{L^3(B_R\setminus B_{R/2})}^3+  \|{b}\|_{L^3(B_R\setminus B_{R/2})}^3 \Big) \notag \\
&\leq  CR^{-\frac4p} \Big( \|u\|_{L^p(B_R\setminus B_{R/2})}^2 +  \|b\|_{L^p(B_R\setminus B_{R/2})}^2 \Big)  \notag \\
&  \hspace{16mm}+ CR^{1-\frac6p}
\Big( \|u\|_{L^p(B_R\setminus B_{R/2})}^3 +  \|b\|_{L^p(B_R\setminus B_{R/2})}^3 \Big).
\end{align}
for all $p \ge 3$.

Hence, for $p\in [3,6]$, we get
\beno
\lim_{R\to \infty} \int_{B_{R/2}}|\nabla u|^2+|\nabla b|^2\,dx = 0,
\eeno
provided $u,b\in L^p(\R^2)$. It follows that $u$ and $b$ are constants, thus $u\equiv b\equiv0$.
Therefore we finish the proof.
\end{proof}

\vspace{4mm}

\noindent
By incorporating with the translation, the estimate \eqref{uniloc} implies the following uniform local estimate

\begin{cor}
\label{le-uni-loc}
For smooth solutions $u,b$ to the MHD equations \eqref{eq:MHD}, we have
\begin{align}
\label{uni-loc}
&\hspace{-8mm} \int_{B_{R/2}(x_0)}|\nabla u|^2+|\nabla b|^2\,dx  \notag \\
&\leq  CR^{-\frac4p} \Big( \|u\|_{L^p(B_{R}(x_0)\setminus B_{R/2}(x_0))}^2 +  \|b\|_{L^p(B_{R}(x_0)\setminus B_{R/2}(x_0))}^2 \Big) \notag\\
&  \hspace{8mm} +CR^{1-\frac6p}
\Big( \|u\|_{L^p(B_{R}(x_0)\setminus B_{R/2}(x_0))}^3 +  \|b\|_{L^p(B_{R}(x_0)\setminus B_{R/2}(x_0))}^3 \Big),
\end{align}
provided $u,b \in L^p(\R^2)$ with $3\leq p\leq \infty$.
\end{cor}

\noindent
In particular, the above lemma says that $\nabla u$ and $\nabla b$ are uniformly locally in $L^2 (\R^2)$, which will be denoted by $u,b \in \dot H^1_{\textup{uloc}}$, by assuming $u,b \in L^p(\R^2)$ for some $p\ge 3$.
From Corollary \ref{le-uni-loc} one easily obtains the following estimate on the growth of the Dirichlet integral,

\begin{cor}
\label{coro-growth}
For smooth solutions $u,b$ to the MHD equations \eqref{eq:MHD},  we have
\[
\int_{B_{R}(x_0)}|\nabla u|^2+|\nabla b|^2\,dx  \les 1+ R^{1-\frac6p},
\]
provided  $u,b \in L^p(\R^2)$ with $3\leq p\leq \infty$.
\end{cor}

\noi
These two properties are of particular importance in what follows.

%
%

\subsection{Case $2\le  p < 4$.}

\begin{proof}[Proof of Theorem \ref{thm2}: case $2\le  p < 4$.]
The argument for this case is similar to that of the previous one. While different treatments are needed to deal with the nonlinear terms $I_3,\cdots,I_6$, and $I_2', I_3'$.
However, the methods to estimate each of these terms are similar and thus we only compute one term,
say $I_4$,
to illustrate the idea.

With the help of (\ref{esti-ws}) and  (\ref{eq:poincare-sobolev}), we have
\begin{align*}
|I_4| & = \Big| \int_{B_\tau\setminus B_{\frac23\tau}}u\cdot\nabla \bar{w} \cdot u \,dx \Big| \\
&\leq\|u\|_{L^4(B_\tau\setminus B_{\frac23\tau})}^2\|\nabla \bar{w}\|_{L^2(B_\tau\setminus B_{\frac23\tau})}\\
&\leq \frac{C}{\tau-\rho} \|{u}\|_{L^2(B_R\setminus B_{R/2})}\left[  \|{u}\|_{L^2(B_{\tau})} \|\nabla{u}\|_{L^2(B_{\tau})}+ \tau^{-2}\|{u}\|_{L^2(B_{\tau})}^2\right]\\
&\leq \frac18\|\nabla{u}\|_{L^2(B_{\tau})}^2+\frac{C}{(\tau-\rho)^2} \|{u}\|_{L^2(B_R\setminus B_{R/2})}^2 \|{u}\|_{L^2(B_{\tau})}^2 \\
& \hspace{6mm} + \frac{C}{\tau^2(\tau-\rho)} \|{u}\|_{L^2(B_R\setminus B_{R/2})}  \|{u}\|_{L^2(B_{\tau})}^2.
\end{align*}

\noi
Similar arguments for all other terms finally lead to
\begin{align*}
f(\rho) & \le \frac12 f(\tau) + \frac{C}{(\tau-\rho)^2} \Big( \|{u}\|_{L^2(B_R)}^2 + \|{b}\|_{L^2(B_R)}^2 \Big)
+ \frac{C}{(\tau-\rho)^2} \Big( \|{u}\|_{L^2(B_R)}^4 + \|{b}\|_{L^2(B_R)}^4 \Big) \\
& \hspace{6mm} + \frac{C}{\tau^2(\tau-\rho)} \Big(  \|{u}\|_{L^2(B_{R})}^3 + \|{b}\|_{L^2(B_R)}^3 \Big),
\end{align*}
where $f(\rho)$ was defined in \eqref{f} and $\frac{3R}4 \le \rho < \tau \le R$.
Then we apply Lemma \ref{iter} to obtain
\begin{align*}
& \hspace{-6mm} \int_{B_{R/2}}|\nabla u|^2+|\nabla b|^2\,dx\\
&\leq  \frac{C}{R^2} \Big(\|{u}\|_{L^2(B_R)}^2+  \|{b}\|_{L^2(B_R)}^2 \Big)
+ \frac{C}{R^3} \Big(\|{u}\|_{L^2(B_R)}^3+  \|{b}\|_{L^2(B_R)}^3 \Big)\\
& \hspace{6mm} + \frac{C}{R^{2}} \Big(\|{u}\|_{L^2(B_R)}^4+  \|{b}\|_{L^2(B_R)}^4 \Big)\\
&\leq  CR^{-\frac4p} \Big(\|{u}\|_{L^p (B_R)}^2+  \|{b}\|_{L^2(B_R)}^2 \Big)
+ CR^{-\frac{6}p} \Big(\|{u}\|_{L^p (B_R)}^3+  \|{b}\|_{L^p (B_R)}^3 \Big)\\
& \hspace{6mm} +CR^{2-\frac8p} \Big(\|{u}\|_{L^p (B_R)}^4+  \|{b}\|_{L^p (B_R)}^4 \Big),
\end{align*}

\noi
which implies the triviality of $u,b$ when $2\le p<4$.
Therefore we complete the proof for this case.\end{proof}

\subsection{Case $6< p < \infty$.} We'll prove that the case can imply the case of $p=\infty$, hence it is sufficient to consider $p=\infty$ in the next subsection. Under the natural scaling, we can assume that  $\|u\|_{L^p(\mathbb{R}^2)}+\|b\|_{L^p(\mathbb{R}^2)}\leq 1$.

Now we turn to the vorticity and current-density equation
(\ref{eq:MHDv-2D}).
As we have seen in the previous subsection, when $p\le 6$, from \eqref{uni-loc} one has a decay estimate on the gradients
\[
\int_{B_R} |\nabla u|^2 + |\nabla b|^2 \,dx = o \left(\frac1R \right),
\]
as $R\to \infty$, from which the theorem follows.
However, this argument fails when $p> 6$ as the left hand side of \eqref{uni-loc} may fail to decay to zero as $R\to \infty$.
To circumvent this difficulty, we exploit the local properties of the solution instead.
To be more precise, by choosing $R = 2$, \eqref{uni-loc} becomes
\beno
\int_{B_1(x_0)}|\nabla u|^2+|\nabla b|^2\,dx  \leq C(p)\left [1  +
 \|u\|_{L^p(B_2(x_0))}^3 +  \|b\|_{L^p((B_2(x_0))}^3\right]\leq C(p) ,
\eeno
from which we shall show $u,b$ are bounded globally.
To this purpose, we shall first prove that $\nabla^2 u$ and $\nabla^2 b$ are uniformly locally $L^2$ bounded

\begin{lem}
\label{le-uni-loc2}
For $p>6$, assume $u,b\in L^p(\R^2)$ are smooth solutions to \eqref{eq:MHD} and
\beno
 \|u\|_{L^p(\mathbb{R}^2)}+\|b\|_{L^p(\mathbb{R}^2)}\leq 1.
\eeno
 Then, we have
\begin{align}
\label{uni-loc2}
\sup_{x_0} \int_{B_1(x_0)}|\nabla^2 u|^2+|\nabla^2 b|^2\,dx & < \infty.
\end{align}
\end{lem}

\begin{proof}
The idea of the proof is similar to that of Lemma \ref{lem:D2u}.
In view of Corollary \ref{le-uni-loc}, we have the local boundedness \eqref{uni-loc}.
Then, there holds
\begin{align}
\label{uni-loc-2}
\int_{B_1(x_0)}|\nabla u|^2+|\nabla b|^2\,dx  \leq C.
\end{align}
Without loss of any generality, we may assume $x_0 = 0$ for simplicity.

Let $\phi$ be defined as in Subsection 4.1 with $R = 2$.  Multiplying both sides of the vorticity and current-density equation
\eqref{eq:MHDv-2D} by $\phi^2 w$ and $\phi^2 h$ respectively and then
integrating over $\R^2$ to get
\begin{align}
\label{I1-I6}
& \hspace{-6mm} \int \phi^2 |\nabla w |^2 \,dx + \int \phi^2 |\nabla h|^2 \,dx \notag\\
& = -  \int \nabla w  \cdot \nabla (\phi^2) w  \,dx
-  \int  \nabla h \cdot \nabla (\phi^2)  h \,dx+ \frac12 \int u \cdot \nabla  (\phi^2) w^2 \,dx  \notag\\
& \hspace{4mm}  + \frac12 \int u \cdot \nabla  (\phi^2)  h^2 \,dx - \int b\cdot \nabla (\phi^2) hw \,dx + \int H h\phi^2 \,dx \notag\\
& \doteq I_1 + \cdots+ I_6.
\end{align}

\noi
In what follows we shall estimate $I_j$ for $j = 1,2,\cdots, 6$ one by one.

For the term $I_1$, by H\"older's inequality and \eqref{uni-loc-2} we have
\begin{align*}
I_1 & \le \frac{C}{\tau-\rho} \| \nabla w \|_{L^2 (B_\tau)} \| w \|_{L^2 (B_2 )} \\
& \le \frac18 \int_{B_{\tau}} |\nabla w|^2 \,dx + \frac {C}{(\tau-\rho)^2} ,
\end{align*}
where $B_\tau \subset B_2$.
For the term $I_2$, similar argument as the case $I_1$ gives
\begin{align*}
I_2 \le \frac{C}{\tau - \rho} \| \nabla h \|_{L^2 (B_\tau)} \| h\|_{L^2 (B_2)} \le \frac18 \int_{B_{\tau}} | \nabla h  |^2  \,dx  + \frac{C}{(\tau-\rho)^2}.
\end{align*}

To estimate the term $I_3$, we set
\[
(w^2)_{B_{\tau}} = \frac1{|B_{\tau}|} \int_{B_{\tau}} w^2 \,dx \quad \text{ and } \quad
 u_{B_{\tau}} = \frac1{|B_{\tau}|} \int_{B_{\tau}} u \,dx ,
\]
be the means of $w^2$ and $u$ over the ball $B_{\tau}$,
we have
\begin{align*}
I_3 & \leq \left|\int u \cdot \nabla w \phi^2 w \,dx\right| = \frac12 \left|\int u \cdot \nabla \big( w^2 - (w^2)_{B_{\tau}} \big)\phi^2 \,dx\right|\\
& \leq \frac{C}{\tau - \rho} \int_{B_{\tau}} \big| w^2 - (w^2)_{B_{\tau}}  \big| |u -  u_{B_{\tau}} | \,dx +  C\frac{| u_{B_{\tau}} |}{\tau - \rho} \int_{B_{\tau}} w^2  \,dx\\
& \leq \varepsilon  \int_{B_{\tau}} \big| w^2 - (w^2)_{B_{\tau}}  \big|^2 \,dx +  \frac{C_\varepsilon}{(\tau - \rho)^2} \int_{B_{\tau}} |u -  u_{B_{\tau}} |^2 \,dx  + C \frac{| u_{B_{\tau}} |}{\tau - \rho} \int_{B_{\tau}} w^2  \,dx
\end{align*}

\noi
where $\varepsilon >0$ will be determined later.
The last two terms can be easily bounded by
\[
  \frac{C_\varepsilon}{(\tau - \rho)^2} \int_{B_{2}} |\nabla u |^2 \,dx  + \frac{C}{\tau-\rho} \int_{B_{2}} w^2 \,dx \int_{B_{2}} |u| \,dx,
\]
where we used Poincar\'e's inequality. For the first term, by Poincar\'e inequality, H\"older's inequality, and Young's inequality, we arrive at
\begin{align*}
 & \hspace{-6mm} \int_{B_{\tau}} \big| w^2 -(w^2)_{B_{\tau}}  \big|^2  \,dx \\
& \leq C \left( \int_{B_{\tau}} \big| \nabla (w^2) \big|  \,dx \right)^{2} \\
& \leq C\int_{B_{\tau}} | \nabla w  |^2  \,dx  \int_{B_{2}} | w  |^2  \,dx  .
\end{align*}
To summary, by choosing $\varepsilon$ small enough and applying \eqref{uni-loc-2} we have
\begin{align}
\label{uni-loc-I2}
I_3 \le  \frac18 \int_{B_{\tau}} | \nabla w  |^2  \,dx  + \frac{C}{\tau-\rho} +  \frac{C}{(\tau-\rho)^2}.
\end{align}

For the term $I_4$, the proof is similar to that of $I_3$, we have
\begin{align*}
I_4 \le  \frac18 \int_{B_{\tau}} | \nabla h  |^2  \,dx  + \frac{C}{\tau-\rho} +  \frac{C}{(\tau-\rho)^2}.
\end{align*}

To bound $I_5$, we use a similar argument as that of $I_3$ but with the following application of Poincar\'e inequality instead
\begin{align*}
 \int_{B_{\tau}} \big| wh -(wh)_{B_{\tau}}  \big|^2  \,dx
 \leq C \int_{B_{\tau}} | \nabla h  |^2  \,dx  \int_{B_{2}} | w  |^2  \,dx +C \int_{B_{\tau}} | \nabla w  |^2  \,dx  \int_{B_{2}} | h  |^2  \,dx .
\end{align*}
One then gets
\begin{align*}
I_5 \le \frac18 \int_{B_{\tau}} | \nabla h  |^2  \,dx +  \frac18 \int_{B_{\tau}} | \nabla w  |^2  \,dx  + \frac{C}{\tau-\rho} +  \frac{C}{(\tau-\rho)^2}.
\end{align*}

Lastly, for the term $I_6$, by using H\"older inequality and Gagliardo-Nirenberg inequality, we have
\begin{align*}
\hspace{-6mm} \int \phi^2 h H \,dx
&\leq \| \phi \nabla b \|^2_{L^4 (B_\tau)} \| \nabla u \|_{L^2 (B_2)} \\
&\leq C \| \phi \nabla b \|_{L^2 (B_\tau)}\| \nabla(\phi \nabla b) \|_{L^2 (B_\tau)}\| w \|_{L^2 (B_2)} \\
&\leq  \frac18\| \nabla h \|_{L^2 (B_\tau)}^2+C+\frac{C}{\tau-\rho},
\end{align*}
where we also used (\ref{eq:embeding inequ}).

By denoting
\[
g(r) = \int_{B_{r}} | \nabla h  |^2  \,dx +  \int_{B_{r}} | \nabla w  |^2  \,dx ,
\]
we finally arrive at
\[
g(\rho) \le \frac12 g(\tau) + C + \frac C{\tau-\rho} +  \frac C{(\tau-\rho)^2}
\]
for all $\frac34 \le \rho < \tau \le 2$.
Then an application of Lemma \ref{iter} yields
\[
 \int_{B_1} | \nabla h  |^2  \,dx +  \int_{B_1} | \nabla w  |^2  \,dx \leq C.
\]
Then the desired bound \eqref{uni-loc2} follows.
\end{proof}

One direct consequence of Lemma \ref{le-uni-loc2} is the boundedness of $u$ and $ b $.

\begin{cor}
\label{bdd}
With the same assumptions as Lemma \ref{le-uni-loc2}, we have
\[
\| u \|_{L^\infty(\R^2; \R^2)} + \| b \|_{L^\infty(\R^2; \R^2)}\leq C(p) < \infty.
\]
\end{cor}

%
%

\subsection{Case $p=\infty$.}

\noi
In Subsection 4.3, it is showed that $u,b \in L^\infty $ provided $ u,b \in L^p (\R^2 ; \R^2)$ for some $p\in (6,\infty)$.
Next we shall start with $u,b \in L^\infty $. At first, we strengthen the local estimate \eqref{uni-loc2} into a global one.
More precisely, we have

\begin{lem}
\label{d2u-bound}
Let smooth solutions $u,b$ to the MHD equations \eqref{eq:MHD} satisfying 
\beno
\| u \|_{L^\infty(\R^2; \R^2)} + \| b \|_{L^\infty(\R^2; \R^2)}\leq 1.
\eeno
Then, there exists an absolute constant $C_*$ such that if
\beno
\|b\|_{L^1(\R^2)}  \leq C_*  \min\{ \mu\nu, \mu^{\frac12}\nu^{\frac32}\} ,
\eeno
there holds
\[
 \int_{\R^2} | \nabla^2 u  |^2  \,dx +  \int_{\R^2} | \nabla^2 b  |^2  \,dx \leq C.
\]
\end{lem}

\begin{proof}

In view of \eqref{eq:embeding inequ}, it suffices to show
\[
\int_{\R^2} | \nabla w  |^2  \,dx +  \int_{\R^2} | \nabla h  |^2  \,dx \les 1.
\]
The proof is a modification of the proof of Lemma \ref{le-uni-loc2}.
As in the proof of Lemma \ref{le-uni-loc2}, we get
\begin{align}
\label{I1-I6}
& \hspace{-6mm} \mu\int \phi^2 |\nabla w |^2 \,dx + \nu\int \phi^2 |\nabla h|^2 \,dx \notag \\
& = -  \int \nabla w  \cdot \nabla \phi^2 w  \,dx
-  \int  \nabla h \cdot \nabla \phi^2  h \,dx+ \frac12 \int u \cdot \nabla  \phi^2 w^2 \,dx  \notag\\
& \hspace{4mm}  + \frac12 \int u \cdot \nabla  \phi^2  h^2 \,dx - \int b\cdot \nabla \phi^2 hw \,dx + \int H \phi^2 h \,dx \notag\\
& = I_1 + \cdots +I_6.
\end{align}

\noi
where $\phi$ is a test function as in the proof of Lemma \ref{le-uni-loc} with $|\nabla \phi| \leq \frac{C}{\tau-\rho}$ and $|\nabla^2 \phi| \leq \frac{C}{(\tau-\rho)^2}$.
We shall show all the above terms are bounded uniformly in $R$.

For the term $I_1$, by Corollary \ref{coro-growth} we have
\[
|I_1| = \frac12 \int  w^2   |\Delta \phi^2 | \,dx \leq  \frac{C}{(\tau-\rho)^2} \int_{T_\tau} w^2 \,dx \leq  \frac{C \tau}{(\tau-\rho)^2},
\]
where $T_\tau=B_{\tau} \setminus B_{\frac23 \tau}.$
For the term $I_3$, since $u\in L^\infty$  and then we have
\[
|I_3| = C \int  w^2   |\nabla \phi^2|  \,dx \leq  \frac{C}{\tau-\rho} \int_{T_\tau} w^2 \,dx \leq \frac{C\tau}{\tau-\rho} .
\]

\noi
For the term $I_2,I_4,I_5$, similar as $I_1$ and $I_3$ we have
\[
|I_2| + |I_4|+ |I_5|   \leq
\frac{C\tau}{\tau-\rho} + \frac{C \tau}{(\tau-\rho)^2}.
\]


Now we turn to the term $I_6$, which is the main difficulty. Without loss of any generality, we may assume $H = \partial_{2} b_2 \partial_{2} u_1$. Applying integration by parts we obtain
\begin{align*}
| I_6 | & = \left| \int \phi^2 h \partial_{2} b_2 \partial_{2} u_1\, dx \right| \\
& \le \left|  \int \partial_{2} \phi^2 h \partial_{2} u_1b_2\,dx \right|
+ \left| \int  \phi^2 \partial_{2} h \partial_{2} u_1 b_2  \,dx \right|
+ \left| \int  \phi^2  h \partial^2_{2} u_1b_2  \,dx \right| \\
& = I_{61} + I_{62} + I_{63}.
\end{align*}

\noi
The first term can be bounded easily by using $b \in L^\infty$ and Corollary \ref{coro-growth},
\[
I_{61} \le \frac{C}{\tau-\rho} \int_{B_\tau} | h \partial_{2} u| \,dx \le \frac{C \tau}{\tau-\rho}.
\]

\noi
The terms $I_{62}$ and $I_{63}$ can be treated in a similar way. Therefore we only
estimate the former,
for which we have
\begin{align*}
& \hspace{-6mm} \int  \phi^2 \partial_{2} h \partial_{2} u_1 b_2  \,dx \\
& =  \int  \partial_{2} h \partial_{2} (u_1\phi) \phi b_2  \,dx - \int  \partial_{2} h u_1 \partial_{2} \phi  b_2 \phi  \,dx  \\
 & = I_{621}+ I_{622}.
\end{align*}

\noi
We notice that the term $I_{622}$ is easy to control,
\begin{align*}
I_{622} & \le \frac{C}{\tau-\rho} \|\nabla h\|_{L^2(B_\tau)}   \|\phi u \|_{L^4} \|b\|_{L^4} \\
& \le \frac{C}{\tau-\rho} \|\nabla h\|_{L^2(B_\tau)}   \| u \|_{L^4 (B_\tau)} \|b\|_{L^1}^{\frac14} \|b\|_{L^\infty}^{\frac34} \\
& \le \frac{C \tau^{\frac12}}{\tau-\rho} \|\nabla h\|_{L^2(B_\tau)}   \| u \|_{L^\infty (B_\tau)} \|b\|_{L^1}^{\frac14} \|b\|_{L^\infty}^{\frac34},
\end{align*}

\noi
which is sufficient for our purpose.
Now we turn to the term $I_{621}$,
\begin{align*}
I_{621} & \le \| \nabla h\|_{L^2(B_\tau)} \| \partial_{2} (u_1 \phi) \phi b_2\|_{L^2} \\
& \le  \| \nabla h\|_{L^2(B_\tau)} \| \nabla (u \phi)\|_{L^4 (\R^2)} \| \phi b\|_{L^4 (\R^2)}.
\end{align*}

\noi
Then we apply the following two Gagliardo-Nirenberg inequalities,
\[
\| \nabla f\|_{L^4(\R^2)} \le \| \nabla^2 f\|_{L^2(\R^2)}^{\frac12} \|f\|_{L^\infty(\R^2)}^{\frac12} ,
\]
and
\[
\|  f\|_{L^4(\R^2)} \le  \| \nabla^2 f\|_{L^2(\R^2)}^{\frac12} \|f\|_{L^1(\R^2)}^{\frac12},
\]
to obtain
\begin{align*}
I_{621} & \le  C \| \nabla h\|_{L^2(B_\tau)}  \left( \| \nabla^2 (u\phi) \|_{L^2 (\R^2)} \| \nabla^2 (b\phi) \|_{L^2 (\R^2)} \right)^{\frac12}  \| \phi b\|_{L^1 (\R^2)}^{\frac12} \| \phi u\|_{L^\infty (\R^2)}^{\frac12} \\
& \le \frac14 \left( \int_{B_\tau} \mu|\nabla h|^2+\nu|\nabla w|^2 \,dx \right)+C+\frac{C\tau}{(\tau-\rho)^2}+\frac{C\tau}{(\tau-\rho)^4},
\end{align*}
provided $\|b\|_{L^1} $ is small enough and we used (\ref{eq:embeding inequ}) and  Corollary \ref{coro-growth}.

Finally, by setting
\[
g(r) = \int_{B_r}   |\nabla h|^2+|\nabla w|^2 \,dx,
\]
we arrive at
\[
g(\rho) \le \frac12 g(\tau) + C + \frac{CR }{\tau-\rho} + \frac{CR}{(\tau-\rho)^4},
\]

\noi
where $\frac{3R}{4}\leq \rho < \tau \leq  R$.

Hence, by Lemma \ref{iter}, we have
\beno
\int_{B_{R/2}} |\nabla w |^2 +   |\nabla h|^2 \,dx\leq C+\frac{C}{{R^3}}.
\eeno
Letting $R\rightarrow\infty$, the proof is complete.
\end{proof}


In fact, by assuming $b$ is small enough in $L^1$ space,  we can conclude that $\nabla^2u$ and $ \nabla^2 b$ are both trivial.
More precisely, we have

\begin{lem}
\label{d2u-tri}
Let $(u,b)$ be a weak solution of the 2D MHD equations (\ref{eq:MHD}) defined over the entire plane.
Assume that 
\beno
\|u\|_{L^\infty(\mathbb{R}^2)}+\|b\|_{L^\infty(\mathbb{R}^2)}\leq 1.
\eeno
There exists a constant $C_*$ such that if
\beno
\|b\|_{L^1(\R^2)}  \leq C_*  \min\{ \mu\nu, \mu^{\frac12}\nu^{\frac32}\} ,
\eeno
then
\[
\nabla^2u \equiv0,\quad  \nabla^2 b \equiv0.
\]
\end{lem}

\begin{proof}
In view of Lemma \ref{d2u-bound}, we may assume $\nabla^2 u, \nabla^2 b \in L^2(\mathbb{R}^2)$ in what follows.
We shall revisit the proof of Lemma \ref{d2u-bound} and show the terms $I_1$ to $I_5$ in \eqref{I1-I6}
vanishes and $I_6$ becomes small as $R$ goes to infinity.

Let $\phi$ be replaced by $\eta_R$ in Lemma \ref{d2u-bound}, then one notices that the terms $I_1$ and $I_2$ tend to zero as $R\to \infty$. The treatments for terms $I_2,I_4,I_5$ are similar, thus we only focus on the term $I_3$. Let $\chi(x)$ be a test function such that
\begin{align*}
\chi_{_R}(x)=\left\{
\begin{aligned}
&1,\quad x\in B_R\setminus B_{R/2},\\
&0, \quad x\in B_{2R}^c\cup B_{R/4}.
\end{aligned}
\right. \end{align*}
and $|\nabla^k \chi | \le \frac{C}{R^k}$. Then
\begin{align*}
I_3&\leq \frac{C}{R} \int_{B_R\setminus B_{R/2}} |u | w^2 \,dx \le \frac{C}{R} \int_{B_R\setminus B_{R/2}} w^2 \,dx \leq \frac{C}{R}\|\nabla (u\chi_{_R})\|_{L^2(T_R)}^2\\
&\leq \frac{C}{R} \|u\|_{L^2(T_R)}\left[\|\chi_{_R}\nabla^2 u\|_{L^2(T_R)}+\|\nabla u |\nabla\chi_{_R}| \|_{L^2(T_R)}+\||u|\nabla^2 \chi_{_R}\|_{L^2(T_R)}\right]\\
&\leq C \|\nabla^2 u\|_{L^2(T_R)}+ \frac{C}{\sqrt{R}},
\end{align*}
where $T_R = B_{2R} \setminus B_{R/4}$, and we used Corollary \ref{coro-growth}.
Obviously, $I_3$ tends to zero as $R\to \infty$ by Lemma \ref{d2u-bound}.

Now we turn to the term $I_6$. Unlike other terms, we will not show $I_6$ goes to zero as $R\to \infty$,
instead we shall show $I_6$ tends to something smaller than the left hand side of \eqref{I1-I6}, which implies
the desired result.

Without loss of any generality, assume $H =  \partial_{2} b_2\partial_{2} u_1$. Therefore,
\begin{align*}
I_{6} & =  \int_{\mathbb{R}^2} \phi h \partial_{2} b_2\partial_{2} u_1\,dx\\
& = - \int_{\mathbb{R}^2} \phi h \partial_{2}^2u_1 b_2\,dx-\int_{\mathbb{R}^2} \phi \partial_{2}h \partial_{2}u_1 b_2 \,dx-\int_{\mathbb{R}^2} \partial_{2}\phi h \partial_{2}u_1 b_2 \,dx \\
& =J_1+J_2+J_3.
\end{align*}

\noi
Then we shall estimate $J_1$, $J_2$, and $J_3$ one by one. For the term $J_1$,
\begin{align*}
J_1&\leq \|\partial_{2}^2 u_1\|_{L^2(\mathbb{R}^2)}\|h\|_{L^4(\mathbb{R}^2)}\|b\|_{L^4(\mathbb{R}^2)}\\
&\leq C \|\partial_{2}^2 u\|_{L^2(\mathbb{R}^2)}\|b\|_{L^\infty(\mathbb{R}^2)}^{\frac12}\|\nabla^2 b\|_{L^2(\mathbb{R}^2)}^{\frac12}\|b\|_{L^1(\mathbb{R}^2)}^{\frac12}\|\nabla^2b\|_{L^2(\mathbb{R}^2)}^{\frac12}\\
&\leq \frac{\mu}{4}\|\nabla^2 u\|_{L^2(\mathbb{R}^2)}+\frac{\nu}{4}\|\nabla^2 b\|_{L^2(\mathbb{R}^2)},
\end{align*}

\noi
provided $\|b\|_{L^1} $ sufficiently small, that is $ \|b\|_{L^1}\leq C_* \mu\nu$
holds for a small $C_*$.
For the term $J_2$, it can be estimated in the same way
\begin{align*}
J_2&\leq \|\partial_{2}h\|_{L^2(\mathbb{R}^2)}\|\partial_2 u\|_{L^4(\mathbb{R}^2)}\|b\|_{L^4(\mathbb{R}^2)}\\
&\leq C \|\partial_{2}^2 b\|_{L^2(\mathbb{R}^2)}\|u\|_{L^\infty(\mathbb{R}^2)}^{\frac12}\|\nabla^2 u\|_{L^2(\mathbb{R}^2)}^{\frac12}\|b\|_{L^1(\mathbb{R}^2)}^{\frac12}\|\nabla^2b\|_{L^2(\mathbb{R}^2)}^{\frac12}\\
&\leq \frac{\mu}{4}\|\nabla^2 u\|_{L^2(\mathbb{R}^2)}+\frac{\nu}{4}\|\nabla^2 b\|_{L^2(\mathbb{R}^2)},
\end{align*}
provided $ \|b\|_{L^1}  \leq C_* \mu^{\frac12}\nu^{\frac32}$.
The term $J_3$ can be dealt with similarly as $I_3$,
which also vanishes as $R\to \infty$. Thus the proof is finished.
\end{proof}

Now we are ready to finish the remaining part $6< p \leq \infty$ of Theorem \ref{thm2}.

\begin{proof}[Proof of Theorem \ref{thm2}: case $6< p \leq \infty$]
For $6< p \leq \infty$, assume that $(u,b)$ are nontrivial and
\ben\label{eq:Lp}
\|u\|_{L^p(\mathbb{R}^2)}+\|b\|_{L^p(\mathbb{R}^2)}= L>0
\een
then consider the scaling solution $(u^\lambda(x),b^\lambda(x))$,
where
\beno
u^\lambda(x)=\lambda u(\lambda x),\quad b^\lambda(x)=\lambda b(\lambda x)
\eeno
Then by scaling property we get
\beno
\|u^\lambda\|_{L^p(\mathbb{R}^2)}+\|b^\lambda\|_{L^p(\mathbb{R}^2)}\leq 1
\eeno
if $\lambda^{-\frac{p-2}{p}}=L$. 
By the assumption of Theorem  \ref{thm2}, we get there exists an absolute constant $C_*$ such that
\beno
L^{\frac{p}{p-2}}\|b\|_{L^1(\R^2)}\leq C_* \min\{\mu\nu,\mu^{\frac12}\nu^{\frac32}\}
\eeno
and hence
\ben\label{eq:condition p}
\|b^{\lambda}\|_{L^1(\R^2)} =\lambda ^{-1}\|b\|_{L^1(\R^2)}\leq C_* \min\{\mu\nu,\mu^{\frac12}\nu^{\frac32}\}
\een
Then it follows from Corollary \ref{bdd} and Lemma \ref{d2u-tri} that
\beno
\nabla^2(u^{\lambda}) \equiv0,\quad  \nabla^2 (b^{\lambda}) \equiv0,
\eeno
which implies $u,b$ are constants, since $u, b \in L^p(\R^2)$. This is a contradiction with (\ref{eq:Lp}). Hence  $(u,b)$ are trivial solutions.

The proof is complete.
\end{proof}

\noindent {\bf Acknowledgments.}
W. Wang was supported by NSFC under grant 11671067,
``the Fundamental Research Funds for the Central Universities", and the China Scholarship Council. Y. Wang was partially supported by NSFC under grant 11771140.

\end{document}